\documentclass{article}

\usepackage{a4wide}

\usepackage{amsthm,amsfonts,amssymb,amsmath,cite}

\newtheorem{theorem}{Theorem}
\newtheorem{corollary}[theorem]{Corollary}
\newtheorem{lemma}[theorem]{Lemma}
\newtheorem{definition}[theorem]{Definition}
\newtheorem{example}[theorem]{Example}
\newtheorem{remark}[theorem]{Remark}

\begin{document}

\title{Noether's Theorem for Nonsmooth Extremals\\
of Variational Problems with Time Delay\thanks{This is a preprint 
of a paper whose final and definite form will be published in \emph{Applicable Analysis}. 
Manuscript submitted 07-Nov-2012; accepted for publication 19-Dec-2012.}}

\author{Gast\~{a}o S. F. Frederico$^{1,2}$\\
\texttt{gastao.frederico@ua.pt}
\and
Tatiana Odzijewicz$^{1}$\\
\texttt{tatianao@ua.pt}
\and
Delfim F. M. Torres$^{1}$\\
\texttt{delfim@ua.pt}}

\date{$^1$\mbox{CIDMA --- Center for Research and Development in Mathematics and Applications,}\\
Department of Mathematics, University of Aveiro,
3810-193 Aveiro, Portugal\\[0.3cm]
$^2$Department of Science and Technology,\\
University of Cape Verde, Praia, Santiago, Cape Verde}

\maketitle

% ------------------------

\begin{abstract}
We obtain a nonsmooth extension of Noether's symmetry theorem
for variational problems with delayed arguments. The result
is proved to be valid in the class of Lipschitz functions,
as long as the delayed Euler--Lagrange extremals are restricted
to those that satisfy the DuBois--Reymond necessary optimality condition.
The important case of delayed variational problems with higher-order
derivatives is considered as well.

\medskip


\medskip


\end{abstract}

% ------------------------

\section{Introduction}

In 1918, Emmy Noether published a paper that strongly influenced
the physics of the 20th century \cite{MR2761345}. She proved
a theorem asserting that if the Lagrangian is invariant
under changes in the coordinate system, then there exists a conserved
quantity along all the Euler--Lagrange extremals. Within the years,
this result has been studied by many authors and generalized in different directions:
see \cite{Bartos,Gastao:PhD:thesis,GF:JMAA:07,GF2010,MR2323264,book:frac,NataliaNoether,ejc}
and references therein. In particular, in the recent paper \cite{GF2012}, Noether's theorem was formulated
for variational problems with delayed arguments. The result is important
because problems with delays play a crucial role in the modeling
of real-life phenomena in various fields of applications \cite{GoKeMa}.
In order to prove Noether's theorem with delays, it was assumed that
admissible functions are $\mathcal{C}^2$-smooth and that Noether's conserved quantity holds
along all $\mathcal{C}^2$-extremals of the Euler--Lagrange equations with time delay \cite{GF2012}.
Here we remark that when one extends Noether's theorem to the biggest class
for which one can derive the Euler--Lagrange equations, i.e., for Lipschitz continuous functions,
then one can find Lipschitz Euler--Lagrange extremals that fail to satisfy
the Noether conserved quantity established in \cite{GF2012}
(see a simple example in Section~\ref{sec:example}).
We show that to formulate Noether's theorem with time delays
for nonsmooth functions, it is enough to restrict the set of
delayed Euler--Lagrange extremals to those that satisfy the
delayed DuBois--Reymond condition. Moreover,
we prove that this result can be generalized
to higher-order variational problems.

The text is organized as follows. In Section~\ref{sec:prelim}
the fundamental problem of variational calculus with delayed arguments is formulated
and a short review of the results for $\mathcal{C}^2$-smooth admissible functions is given.
In Section~\ref{sec:example} we show, through an example, that nonsmooth
Euler--Lagrange delayed extremals may fail to satisfy Noether's constants of motion \cite{GF2012}.
The main contributions of the paper appear in Sections~\ref{sec:MR} and \ref{sec:MRHO}:
we prove a Noether symmetry theorem with time delay for Lipschitz functions (Theorem~\ref{theo:tnnd}),
Euler--Lagrange and DuBois--Reymond optimality type conditions
for nonsmooth higher-order variational problems with delayed arguments
(Theorems~\ref{Thm:ELdeordm} and~\ref{theo:cDRifm}, respectively),
and a delayed higher-order Noether's symmetry theorem (Theorem~\ref{thm:Noether}).

% ------------------------

\section{Preliminaries}
\label{sec:prelim}

In this section we review necessary results on the calculus of variations with time delay.
For more on variational problems with delayed arguments we refer the reader to
\cite{Basin:book,Bok,GoKeMa,DH:1968,Kra,Kharatishvili,Ros}.

The fundamental problem consists of minimizing a functional
\begin{equation}
\label{Pe}
J^{\tau}[q(\cdot)] = \int_{t_{1}}^{t_{2}}
L\left(t,q(t),\dot{q}(t),q(t-\tau),\dot{q}(t-\tau)\right) dt
\end{equation}
subject to boundary conditions
\begin{equation}
\label{Pe2}
q(t)=\delta(t)~\textnormal{ for }~t\in[t_{1}-\tau,t_{1}]~\textnormal{ and }~q(t_2)=q_{t_2}.
\end{equation}
We assume that the Lagrangian $L :[t_1,t_2] \times \mathbb{R}^{4n} \rightarrow \mathbb{R}$, $n\in\mathbb{N}$,
is a $\mathcal{C}^{2}$-function with respect to all its arguments,
admissible functions $q(\cdot)$ are $\mathcal{C}^2$-smooth,
$t_{1}< t_{2}$ are fixed in $\mathbb{R}$,
$\tau$ is a given positive real number such that $\tau<t_{2}-t_{1}$,
and $\delta$ is a given piecewise smooth function on $[t_1-\tau,t_1]$.
Throughout the text, $\partial_{i}L$ denotes the partial derivative of $L$
with respect to its $i$th argument, $i=1,\dots,5$. For convenience of notation,
we introduce the operator $[\cdot]_{\tau}$ defined by
$$
[q]_{\tau}(t)=(t,q(t),\dot{q}(t),q(t-\tau),\dot{q}(t-\tau)).
$$

The next theorem gives a necessary optimality condition of Euler--Lagrange type
for \eqref{Pe}--\eqref{Pe2}.

\begin{theorem}[Euler--Lagrange equations with time delay \cite{DH:1968}]
\label{th:EL1}
If $q(\cdot)\in\mathcal{C}^2$ is a minimizer for problem \eqref{Pe}--\eqref{Pe2},
then $q(\cdot)$ satisfies the following Euler--Lagrange equations with time delay:
\begin{equation}
\label{EL1}
\begin{cases}
\frac{d}{dt}\left\{\partial_{3}L[q]_{\tau}(t)+
\partial_{5}L[q]_{\tau}(t+\tau)\right\}
=\partial_{2}L[q]_{\tau}(t)+\partial_{4}L[q]_{\tau}(t+\tau),
\quad t_{1}\leq t\leq t_{2}-\tau,\\
\frac{d}{dt}\partial_{3}L[q]_{\tau}(t) =\partial_{2}L[q]_{\tau}(t),
\quad t_{2}-\tau\leq t\leq t_{2}.
\end{cases}
\end{equation}
\end{theorem}

\begin{remark}
\label{re:EL}
If one extends the set of admissible functions in problem
\eqref{Pe}--\eqref{Pe2} to the class of Lipschitz continuous functions,
then the Euler--Lagrange equations \eqref{EL1} remain valid.
This result is obtained from our Corollary~\ref{cor:16}
by choosing $m = 1$.
\end{remark}

\begin{definition}[Extremals]
\label{def:scale:ext}
The solutions $q(\cdot)$ of the Euler--Lagrange equations
\eqref{EL1} with time delay are called extremals.
\end{definition}

\begin{definition}[Invariance of \eqref{Pe}]
\label{def:invnd}
Consider the following $s$-parameter group of infinitesimal transformations:
\begin{equation}
\label{eq:tinf}
\begin{cases}
\bar{t} = t + s\eta(t,q) + o(s) \, ,\\
\bar{q}(t) = q(t) + s\xi(t,q) + o(s),
\end{cases}
\end{equation}
where $\eta\in \mathcal{C}^1(\mathbb{R}^{n+1},\mathbb{R})$
and $\xi\in \mathcal{C}^1(\mathbb{R}^{n+1},\mathbb{R}^n)$.
Functional \eqref{Pe} is said to be invariant
under \eqref{eq:tinf} if
\begin{multline*}
\label{eq:invnd}
0 = \frac{d}{ds}
\int_{\bar{t}(I)} L\left(t+s\eta(t,q(t))+o(s),q(t)+s\xi(t,q(t))+o(s),
\frac{\dot{q}(t)+s\dot{\xi}(t,q(t))}{1+s\dot{\eta}(t,q(t))},\right.\\
\left.q(t-\tau)+s\xi(t-\tau,q(t-\tau))+o(s),\frac{\dot{q}(t-\tau)
+s\dot{\xi}(t-\tau,q(t-\tau))}{1+s\dot{\eta}(t-\tau,q(t-\tau))}\right)
(1+s\dot{\eta}(t,q(t))) dt \Biggl.\Biggr|_{s=0}
\end{multline*}
for any  subinterval $I \subseteq [t_1,t_2]$.
\end{definition}

\begin{definition}[Constant of motion/conservation law with time delay]
\label{def:leicond}
We say that a quantity
$C(t,t+\tau,q(t),q(t-\tau),q(t+\tau),\dot{q}(t),\dot{q}(t-\tau),\dot{q}(t+\tau))$ is
a constant of motion with time delay $\tau$ if
\begin{equation}
\label{eq:conslaw:td1}
\frac{d}{dt} C(t,t+\tau,q(t),q(t-\tau),q(t+\tau),\dot{q}(t),\dot{q}(t-\tau),\dot{q}(t+\tau))= 0
\end{equation}
along all the extremals $q(\cdot)$ (\textrm{cf.} Definition~\ref{def:scale:ext}).
The equality \eqref{eq:conslaw:td1} is then a conservation law with time delay.
\end{definition}

Next theorem extends the DuBois--Reymond necessary
optimality condition
to problems of the calculus of variations with time delay.

\begin{theorem}[DuBois--Reymond necessary conditions with time delay \cite{GF2012}]
\label{theo:cdrnd}
If $q(\cdot)\in\mathcal{C}^2$ is an extremal of functional
\eqref{Pe} subject to \eqref{Pe2}, then the following conditions are satisfied:
\begin{equation}
\label{eq:cdrnd}
\begin{cases}
\frac{d}{dt}\left\{L[q]_{\tau}(t)-\dot{q}(t)\cdot(\partial_{3} L[q]_{\tau}(t)
+\partial_{5} L[q]_{\tau}(t+\tau))\right\} = \partial_{1} L[q]_{\tau}(t),\quad
t_1\leq t\leq t_{2}-\tau,\\
\frac{d}{dt}\left\{L[q]_{\tau}(t)
-\dot{q}(t)\cdot\partial_{3} L[q]_{\tau}(t)\right\}
=\partial_{1} L[q]_{\tau}(t),\quad t_2-\tau\leq t\leq t_{2}\,.
\end{cases}
\end{equation}
\end{theorem}

\begin{remark}
If we assume that admissible functions in problem \eqref{Pe}--\eqref{Pe2}
are Lipschitz continuous, then one can show that the DuBois--Reymond
necessary conditions with time delay \eqref{eq:cdrnd}
are still valid (cf. Corollary~\ref{cor:DR:m1}).
\end{remark}

Theorem~\ref{theo:tnnd1} establishes an extension of Noether's
theorem to problems of the calculus of variations with time delay.

\begin{theorem}[Noether's symmetry theorem with time delay \cite{GF2012}]
\label{theo:tnnd1}
If functional \eqref{Pe} is invariant in the
sense of Definition~\ref{def:invnd}, then the quantity
$C(t,t+\tau,q(t),q(t-\tau),q(t+\tau),\dot{q}(t),\dot{q}(t-\tau),\dot{q}(t+\tau))$
defined by
\begin{multline}
\label{eq:tnnd1}
\left(\partial_{3} L[q]_{\tau}(t)
+\partial_{5} L[q]_{\tau}(t+\tau)\right)\cdot\xi(t,q(t))\\
+\Bigl(L[q]_{\tau}(t)-\dot{q}(t)\cdot(\partial_{3} L[q]_{\tau}(t)
+\partial_{5} L[q]_{\tau}(t+\tau))\Bigr)\eta(t,q(t))
\end{multline}
for $t_1\leq t\leq t_{2}-\tau$ and by
\begin{equation}
\label{eq:tnnd2}
\partial_{3} L[q]_{\tau}(t)\cdot\xi(t,q(t))
+\Bigl(L[q]_{\tau}(t)-\dot{q}(t)\cdot\partial_{3} L[q]_{\tau}(t)\Bigr)\eta(t,q(t))
\end{equation}
for $t_2-\tau\leq t\leq t_{2}\,,$ is a constant of motion with time delay
(\textrm{cf.} Definition~\ref{def:leicond}).
\end{theorem}

% ------------------------

\section{Nonsmooth Euler--Lagrange extremals may fail
to satisfy Noether's conservation laws with time delay}
\label{sec:example}

Consider the problem of the calculus of variations with time delay
\begin{equation}
\label{eq:ex}
\begin{gathered}
J^{1}[q(\cdot)]=\int_0^3\left(\dot{q}(t) + \dot{q}(t-1) \right)^2dt \longrightarrow \min,\\
q(t)=-t \, ,~-1\leq t\leq 0, \quad q(3)=1,
\end{gathered}
\end{equation}
in the class of functions $q(\cdot)\in Lip\left([-1,3];\mathbb{R}\right)$.
From Theorem~\ref{th:EL1} (see Remark~\ref{re:EL}),
one obtains that any solution to problem \eqref{eq:ex} must satisfy
\begin{equation}
\label{eq:ex:EL1}
2\dot{q}(t)+\dot{q}(t-1)+\dot{q}(t+1)=c_1,
\quad 0\leq t\leq 2,
\end{equation}
\begin{equation}
\label{eq:ex:EL2}
\dot{q}(t)+\dot{q}(t-1)=c_2,
\quad 2\leq t\leq 3,
\end{equation}
where $c_1$ and $c_2$ are constants. Because functional $J^{1}$
of problem \eqref{eq:ex} is autonomous, we have invariance, in the sense of Definition~\ref{def:invnd},
with $\eta\equiv 1$ and $\xi\equiv 0$. Simple calculations show that Noether's constant of motion
with time delay \eqref{eq:tnnd1}--\eqref{eq:tnnd2} coincides with the DuBois--Reymond condition \eqref{eq:cdrnd}:
\begin{equation}
\label{eq:ex:DBR1}
\left(\dot{q}(t)+\dot{q}(t-1)\right)^2-2\dot{q}(t)\left(2\dot{q}(t)
+\dot{q}(t-1)+\dot{q}(t+1)\right)=c_3,\quad 0\leq t\leq 2,
\end{equation}
\begin{equation}
\label{eq:ex:DBR2}
\dot{q}(t)^2-\dot{q}(t-1)^2=c_4,\quad 2\leq t\leq 3,
\end{equation}
where $c_3$ and $c_4$ are constants. One can easily check that function
\begin{equation}
\label{ex:sol}
q(t)=
\begin{cases}
-t & ~\textnormal{for}~ -1< t\leq 0\\
t & ~\textnormal{for}~ 0< t\leq 2\\
-t+4 & ~\textnormal{for}~ 2< t\leq 3
\end{cases}
\end{equation}
satisfies \eqref{eq:ex:EL1}--\eqref{eq:ex:EL2} with $c_1=2$ and $c_2=0$,
but does not satisfy \eqref{eq:ex:DBR1}--\eqref{eq:ex:DBR2}:
for $0< t\leq 1$ constant $c_3$ should be $-4$ and for $1< t\leq 2$ constant $c_3$ should be $0$.
We conclude that nonsmooth solutions of Euler--Lagrange equations \eqref{EL1} do not preserve
Noether's quantity defined by \eqref{eq:tnnd1}--\eqref{eq:tnnd2} and one needs
to restrict the set of Euler--Lagrange extremals. In Section~\ref{sec:MR} we show that
it is enough to restrict the Euler--Lagrange extremals to those that satisfy
the DuBois--Reymond necessary condition \eqref{eq:cdrnd}.

% ------------------------

\section{Noether's theorem with time delay for Lipschitz functions}
\label{sec:MR}

The notion of invariance given in Definition~\ref{def:invnd} can be extended up to
an exact differential.

\begin{definition}[Invariance up to a gauge-term]
\label{def:invndLIP}
We say that functional \eqref{Pe} is invariant under the
$s$-parameter group of infinitesimal transformations \eqref{eq:tinf}
up to the gauge-term $\Phi$ if
\begin{multline}
\label{eq:invndLIP}
\int_{I} \dot{\Phi}[q]_{\tau}(t)dt = \frac{d}{ds}
\int_{\bar{t}(I)} L\left(t+s\eta(t,q(t))+o(s),q(t)+s\xi(t,q(t))+o(s),
\frac{\dot{q}(t)+s\dot{\xi}(t,q(t))}{1+s\dot{\eta}(t,q(t))},\right.\\ \left.
q(t-\tau)+s\xi(t-\tau,q(t-\tau))+o(s),\frac{\dot{q}(t-\tau)
+s\dot{\xi}(t-\tau,q(t-\tau))}{1+s\dot{\eta}(t-\tau,q(t-\tau))}\right)
(1+s\dot{\eta}(t,q(t))) dt\Biggr|_{s=0}
\end{multline}
for any  subinterval $I \subseteq [t_1,t_2]$ and for all
$q(\cdot)\in Lip\left([t_1-\tau,t_2];\mathbb{R}^n\right)$.
\end{definition}

\begin{lemma}[Necessary condition of invariance]
\label{thm:CNSI:SCV}
If functional \eqref{Pe} is invariant up to $\Phi$ in the sense
of Definition~\ref{def:invndLIP}, then
\begin{multline}
\label{eq:cnsind1}
\int_{t_1}^{t_2-\tau}\Bigl[-\dot{\Phi}[q]_{\tau}(t)+\partial_{1}
L[q]_{\tau}(t)\eta(t,q) +\left(\partial_{2}
L[q]_{\tau}(t)+\partial_4 L[q]_{\tau}(t+\tau)\right)\cdot\xi(t,q)\\
+\left(\partial_{3}L[q]_{\tau}(t)
+\partial_5L[q]_{\tau}(t+\tau)\right)\cdot\left(\dot{\xi}(t,q)
-\dot{q}(t)\dot{\eta}(t,q)\right)
+ L[q]_{\tau}(t)\dot{\eta}(t,q)\Bigr]dt = 0
\end{multline}
for $t_1\leq t\leq t_{2}-\tau$ and
\begin{multline}
\label{eq:cnsind2}
\int_{t_2-\tau}^{t_2}\Bigl[-\dot{\Phi}[q]_{\tau}(t)
+\partial_{1}L[q]_{\tau}(t)\eta(t,q)
+\partial_{2}L[q]_{\tau}(t)\cdot\xi(t,q)\\
+\partial_{3}L[q]_{\tau}(t)\cdot\left(\dot{\xi}(t,q)
-\dot{q}(t)\dot{\eta}(t,q)\right)+L[q]_{\tau}(t)\dot{\eta}(t,q)\Bigr]dt =0
\end{multline}
for $t_2-\tau\leq t\leq t_{2}$.
\end{lemma}

\begin{proof}
Without loss of generality, we take $I=[t_1,t_2]$.
Then, \eqref{eq:invndLIP} is equivalent to
\begin{equation}
\label{eq:cnsind3}
\begin{split}
\int_{t_1}^{t_2} \Bigl[ &-\dot{\Phi}[q]_{\tau}(t)
+\partial_{1}L[q]_{\tau}(t)\eta(t,q)
+\partial_{2}L[q]_{\tau}(t)\cdot\xi(t,q)\\
&+\partial_{3}L[q]_{\tau}(t)\cdot\left(\dot{\xi}(t,q)
-\dot{q}(t)\dot{\eta}(t,q)\right)+L[q]_{\tau}(t)\dot{\eta}(t,q)\Bigr]dt\\
&+\int_{t_1}^{t_2}\Bigl[\partial_{4}
L[q]_{\tau}(t)\cdot\xi(t-\tau,q(t-\tau))\\
&+\partial_{5}L[q]_{\tau}(t)\cdot\left(\dot{\xi}(t-\tau,q(t-\tau))
-\dot{q}(t-\tau)\dot{\eta}(t-\tau,q(t-\tau))\right)\Bigr]dt= 0.
\end{split}
\end{equation}
Performing a linear change of variables $t=\sigma+\tau$ in the last integral
of \eqref{eq:cnsind3}, and keeping in mind that $L[q]_{\tau}(t)\equiv 0$ on
$[t_1-\tau,t_1]$, equation \eqref{eq:cnsind3} becomes
\begin{equation}
\label{eq:cnsind}
\begin{split}
\int_{t_1}^{t_2-\tau}\Bigl[&-\dot{\Phi}[q]_{\tau}(t)+\partial_{1}
L[q]_{\tau}(t)\eta(t,q) +\left(\partial_{2}
L[q]_{\tau}(t)+\partial_4 L[q]_{\tau}(t+\tau)\right)\cdot\xi(t,q)\\
&+\left(\partial_{3}L[q]_{\tau}(t)+\partial_5
L[q]_{\tau}(t+\tau)\right)\cdot\left(\dot{\xi}(t,q)
-\dot{q}(t)\dot{\eta}(t,q)\right)
+L[q]_{\tau}(t)\dot{\eta}(t,q)\Bigr]dt\\
&+ \int_{t_2-\tau}^{t_2}\Bigl[-\dot{\Phi}[q]_{\tau}(t)+\partial_{1}
L[q]_{\tau}(t)\eta(t,q) +\partial_{2}
L[q]_{\tau}(t)\cdot\xi(t,q)\\
&+\partial_{3}L[q]_{\tau}(t)\cdot\left(\dot{\xi}(t,q)
-\dot{q}(t)\dot{\eta}(t,q)\right)+L[q]_{\tau}(t)\dot{\eta}(t,q)\Bigr]dt = 0.
\end{split}
\end{equation}
Taking into consideration that \eqref{eq:cnsind} holds for an arbitrary subinterval
$I \subseteq [t_1,t_2]$, equations \eqref{eq:cnsind1} and \eqref{eq:cnsind2} hold.
\end{proof}

\begin{theorem}[Noether's symmetry theorem with time delay for Lipschitz functions]
\label{theo:tnnd}
If functional \eqref{Pe} is invariant up to $\Phi$ in the
sense of Definition~\ref{def:invndLIP}, then the quantity
$C(t,t+\tau,q(t),q(t-\tau),q(t+\tau),\dot{q}(t),\dot{q}(t-\tau),\dot{q}(t+\tau))$
defined by
\begin{multline}
\label{eq:tnnd}
-\Phi[q]_{\tau}(t)+\left(\partial_{3} L[q]_{\tau}(t)
+\partial_{5} L[q]_{\tau}(t+\tau)\right)\cdot\xi(t,q(t))\\
+\Bigl(L[q]_{\tau}-\dot{q}(t)\cdot(\partial_{3} L[q]_{\tau}(t)
+\partial_{5} L[q]_{\tau}(t+\tau))\Bigr)\eta(t,q(t))
\end{multline}
for $t_1\leq t\leq t_{2}-\tau$ and by
\begin{equation}
\label{eq:Noeth}
-\Phi[q]_{\tau}(t)+\partial_{3} L[q]_{\tau}(t)\cdot\xi(t,q(t))
+\Bigl(L[q]_{\tau}-\dot{q}(t)\cdot\partial_{3} L[q]_{\tau}(t)\Bigr)\eta(t,q(t))
\end{equation}
for $t_2-\tau\leq t\leq t_{2}\,,$ is a constant of motion with time delay
along any $q(\cdot)\in Lip\left([t_1-\tau,t_2];\mathbb{R}^n\right)$ satisfying both \eqref{EL1} and
\eqref{eq:cdrnd}, i.e., along any Lipschitz Euler--Lagrange extremal that is also a Lipschitz
DuBois--Reymond extremal.
\end{theorem}

\begin{proof}
We prove the theorem in the interval $t_1\leq t\leq t_{2}-\tau$.
The proof is similar for the interval $t_2-\tau\leq t\leq t_{2}$.
Noether's constant of motion with time delay \eqref{eq:tnnd} follows by
using in the interval $t_1\leq t\leq t_{2}-\tau$ the
DuBois--Reymond condition with time delay \eqref{eq:cdrnd}
and the Euler--Lagrange equation with time delay \eqref{EL1}
into the necessary condition of invariance \eqref{eq:cnsind1}:
\begin{equation*}
\begin{split}
0&=\int_{t_1}^{t_2-\tau}\Bigl[-\dot{\Phi}[q]_{\tau}(t)+\partial_{1}
L[q]_{\tau}(t)\eta(t,q)+\left(\partial_{2}
L[q]_{\tau}(t)+\partial_4 L[q]_{\tau}(t+\tau)\right)\cdot\xi(t,q)\\
&\quad +\left(\partial_{3}L[q]_{\tau}(t)+\partial_5
L[q]_{\tau}(t+\tau)\right)\cdot\left(\dot{\xi}(t,q)
-\dot{q}(t)\dot{\eta}(t,q)\right)+L[q]_{\tau}(t)\dot{\eta}(t,q)\Bigr]dt\\
&= \int_{t_1}^{t_2-\tau}\Bigl[\frac{d}{dt}-\Phi[q]_{\tau}(t)+\left(\partial_{3}
L[q]_{\tau}(t)+\partial_5L[q]_{\tau}(t+\tau)\right)\cdot\xi(t,q)\\
&\quad +\left(\partial_{3}
L[q]_{\tau}(t)+\partial_5L[q]_{\tau}(t+\tau)\right)\cdot\dot{\xi}(t,q)\\
&\quad +\frac{d}{dt}\left\{L[q]_{\tau}(t)-\dot{q}(t)\cdot(\partial_{3} L[q]_{\tau}(t)
+\partial_{5} L[q]_{\tau}(t+\tau))\right\}\eta(t,q)\\
&\quad +\left\{L[q]_{\tau}(t)-\dot{q}(t)\cdot(\partial_{3} L[q]_{\tau}(t)
+\partial_{5} L[q]_{\tau}(t+\tau))\right\}\dot{\eta}(t,q)\Bigr]dt,
\end{split}
\end{equation*}
that is,
\begin{multline}
\label{eq:cnsind11}
\int_{t_1}^{t_2-\tau}\frac{d}{dt}\Bigl[-\Phi[q]_{\tau}(t)
+\left(\partial_{3} L[q]_{\tau}(t)
+\partial_{5} L[q]_{\tau}(t+\tau)\right)\cdot\xi(t,q(t))\\
+\Bigl(L[q]_{\tau}(t)-\dot{q}(t)\cdot(\partial_{3} L[q]_{\tau}(t)
+\partial_{5} L[q]_{\tau}(t+\tau))\Bigr)\eta(t,q(t))\Bigr]dt = 0.
\end{multline}
Taking into consideration that \eqref{eq:cnsind11} holds for
any subinterval $I\subseteq [t_1,t_2]$, we conclude that
\begin{multline*}
-\Phi[q]_{\tau}(t)+\left(\partial_{3} L[q]_{\tau}(t)
+\partial_{5} L[q]_{\tau}(t+\tau)\right)\cdot\xi(t,q(t))\\
+\Bigl(L[q]_{\tau}(t)-\dot{q}(t)\cdot(\partial_{3} L[q]_{\tau}(t)
+\partial_{5} L[q]_{\tau}(t+\tau))\Bigr)\eta(t,q(t))=\text{constant}.
\end{multline*}
\end{proof}

\begin{example}
Consider problem \eqref{eq:ex}. Function
$q(\cdot)\in Lip\left([-1,3];\mathbb{R}^n\right)$ defined by
\begin{equation}
\label{eq:ext:ex:22b}
q(t)=
\begin{cases}
-t & ~\textnormal{for}~ -1< t\leq 0\\
t & ~\textnormal{for}~ 0< t\leq 1\\
-t+2 & ~\textnormal{for}~ 1< t\leq 2\\
t-2 & ~\textnormal{for}~ 2< t\leq 3
\end{cases}
\end{equation}
is an Euler--Lagrange extremal, i.e., satisfies \eqref{eq:ex:EL1}--\eqref{eq:ex:EL2}, but,
in contrast with \eqref{ex:sol}, is also a DuBois--Reymond extremal, i.e.,
satisfies \eqref{eq:ex:DBR1}--\eqref{eq:ex:DBR2}. Theorem~\ref{theo:tnnd} asserts the validity
of Noether's constant of motion, which is here easily verified:
\eqref{eq:tnnd}--\eqref{eq:Noeth} holds along \eqref{eq:ext:ex:22b}
with $\Phi\equiv 0$, $\eta\equiv 1$, and $\xi\equiv 0$.
\end{example}

% ------------------------

\section{Nonsmooth higher-order Noether's theorem for problems of the calculus of variations with time delay}
\label{sec:MRHO}

Let $\mathbb{W}^{k,p}$, $k\geq 1$, $1\leq p \leq \infty$, denote the
class of functions that are absolutely continuous with their derivatives
up to order $k-1$, the $k$th derivative belonging to $L^p$. With this notation,
the class $Lip$ of Lipschitz functions is represented by $\mathbb{W}^{1,\infty}$.
We now extend previous results to problems with higher-order derivatives.

% ------------------------

\subsection{Higher-order Euler--Lagrange and DuBois--Reymond optimality conditions with time delay}

Let $m\in\mathbb{N}$ and $q^{(i)}(t)$ denote the $i$th derivative of $q(t)$, $i=0,\dots,m$,
with $q^{(0)}(t)=q(t)$. For simplicity of notation, we introduce
the operator $[\cdot]^m_{\tau}$ by
$$
[q]^m_{\tau}(t) := \Bigl(t,q(t),\dot{q}(t),
\ldots,q^{(m)}(t),\\q(t-\tau),\dot{q}(t-\tau),
\ldots,q^{(m)}(t-\tau)\Bigr).
$$
Consider the following higher-order variational problem with time delay:
to minimize
\begin{equation}
\label{Pm}
J^{\tau}_{m}[q(\cdot)] =\int_{t_1}^{t_2}
L[q]^m_{\tau}(t)
dt
\end{equation}
subject to the boundary conditions \eqref{Pe2} and
$q^{(i)}(t_2)=q_{t_2}^i,~i=1,\dots,m-1$.
The Lagrangian $L :[t_1,t_2] \times \mathbb{R}^{2 n (m+1)} \rightarrow \mathbb{R}$
is assumed to be a $\mathcal{C}^{m+1}$-function with respect to all its arguments,
admissible functions $q(\cdot)$ are assumed to be $\mathbb{W}^{m,\infty}$,
$t_{1}< t_{2}$ are fixed in $\mathbb{R}$,
$\tau$ is a given positive real number such that $\tau<t_{2}-t_{1}$,
and $q_{t_2}^i$ are given vectors in $\mathbb{R}^n$, $i=1,\dots,m-1$.

\begin{remark}
When $m=1$ functional \eqref{Pm} reduces to \eqref{Pe},
i.e., $J^{\tau}_{1} = J^{\tau}$.
\end{remark}

A variation of $q\in \mathbb{W}^{m,\infty}\left([t_1-\tau,t_2],
\mathbb{R}^{n}\right)$ is another function in the set
$\mathbb{W}^{m,\infty}\left([t_1-\tau,t_2], \mathbb{R}^{n}\right)$ of
the form $q+ \varepsilon h$, with $h \in \mathbb{W}^{m,\infty}\left([t_1-\tau,t_2],
\mathbb{R}^{n}\right)$, such that $h^{(i)}(t_2)=0,\,i=0,\ldots,m$, $h(t)=0$
if $t \in[t_{1}-\tau,t_{1}]$, and $\varepsilon$ a small real positive number.

\begin{definition}[Extremal of \eqref{Pm}]
\label{df1}
We say that $q$ is an extremal of the delayed funcional
\eqref{Pm} if for any $h(\cdot)\in \mathbb{W}^{m,\infty}\left([t_1-\tau,t_2],
\mathbb{R}^{n}\right)$ such that $h^{(i)}(t_2)=0,\,i=0,\ldots,m-1$,
and $h(t)=0$, $t \in[t_{1}-\tau,t_{1}]$, the following equation holds:
\begin{equation*}
\frac{d}{d\varepsilon}\left.J_m^{\tau}[q + \varepsilon
h]\right|_{\varepsilon = 0}=0.
\end{equation*}
\end{definition}

\begin{theorem}[Higher-order Euler--Lagrange equations with time delay in integral form]
\label{Thm:ELdeordm}
If $q(\cdot)\in\mathbb{W}^{m,\infty}\left([t_1-\tau,t_2],
\mathbb{R}^{n}\right)$ is an extremal of functional \eqref{Pm},
then $q(\cdot)$ satisfies the following
higher-order Euler--Lagrange integral equations  with time delay:
\begin{multline}
\label{eq:ELdeordmInt}
\sum_{i=0}^{m}(-1)^{m-i-1}\Biggl(\underbrace{\int_{t_2-\tau}^{t}
\int_{t_2-\tau}^{s_1}\dots \int_{t_2-\tau}^{s_{m-i-1}}}_{m-i~\textnormal{times}}
\Bigl(\partial_{i+2}
L[q]^m_{\tau}(s_{m-i})\\
+\partial_{i+m+3} L[q]^m_{\tau}(s_{m-i}+\tau)\Bigr)ds_{m-i}\dots ds_2 ds_1\Biggr)=p(t)
\end{multline}
for $t_1\leq t\leq t_{2}-\tau$ and
\begin{equation}
\label{eq:ELdeordmInt1}
\sum_{i=0}^{m}(-1)^{m-i-1}\Biggl(\underbrace{\int_{t_2-\tau}^{t}
\int_{t_2-\tau}^{s_1}\dots
\int_{t_2-\tau}^{s_{m-i-1}}}_{m-i~\textnormal{times}}\Bigl(\partial_{i+2}
L[q]^m_{\tau}(t)\Bigr)ds_{m-i}\dots ds_2 ds_1\Biggr) =p(t)
\end{equation}
for $t_{2}-\tau\leq t \leq t_{2}$, where $p(t)$ is a polynomial of order $m-1$, i.e.,
$p(t)=c_0+c_1t+\dots +c_{m-1}t^{m-1}$
for some constants $c_i\in\mathbb{R}$, $i=0,\dots,m-1$.
\end{theorem}

\begin{proof}
Assume that $q(\cdot)\in\mathbb{W}^{m,\infty}\left([t_1-\tau,t_2],
\mathbb{R}^{n}\right)$ is an extremal of functional \eqref{Pm}.
According to the Definition~\ref{df1}, for any $h(\cdot)\in \mathbb{W}^{m,\infty}\left([t_1-\tau,t_2],
\mathbb{R}^{n}\right)$ such that
$h^{(i)}(t_2)=0,\,i=0,\ldots,m-1$, and $h(t)=0$, $t \in[t_{1}-\tau,t_{1}]$, we have
\begin{equation}
\label{pel}
\int_{t_1}^{t_2}\left(\sum_{i=0}^{m}\partial_{i+2}
L[q]^m_{\tau}(t)\cdot h^{(i)}(t)+\sum_{i=0}^{m}\partial_{i+m+3}
L[q]^m_{\tau}(t)\cdot h^{(i)}(t-\tau)\right)dt = 0.
\end{equation}
Performing the linear change of variables $t=\sigma+\tau$ in the last term of integral
\eqref{pel}, and using the fact that $h(t)=0$ if $t \in[t_{1}-\tau,t_{1}]$, \eqref{pel} becomes
\begin{equation}
\label{pel1}
\int_{t_1}^{t_2}\left(\sum_{i=0}^{m}\partial_{i+2}
L[q]^m_{\tau}(t)\cdot h^{(i)}(t)\right)dt
+\int_{t_1}^{t_2-\tau}\left(\sum_{i=0}^{m}\partial_{i+m+3}
L[q]^m_{\tau}(t+\tau)\cdot h^{(i)}(t)\right)dt = 0.
\end{equation}
By repeated integration by parts one has
\begin{multline}
\label{eq:identity1}
\sum\limits_{i=0}^{m} \int_{t_1}^{t_2} \partial_{i+2}L[q]^m_{\tau}(t)\cdot h^{(i)}(t) dt\\
=\sum\limits_{i=0}^{m}\Biggl\{\Biggl[\sum\limits_{j=1}^{m-i}(-1)^{j+1}h^{(i+j-1)}(t)
\cdot\Biggl(\underbrace{\int_{t_2-\tau}^{t}\int_{t_2-\tau}^{s_1}
\dots \int_{t_2-\tau}^{s_{j-1}}}_{j~\textnormal{times}}\Bigl(\partial_{i+2}
L[q]^m_{\tau}(s_j)\Bigr)ds_{j}\dots ds_2 ds_1\Biggr)\Biggr]_{t_1}^{t_2}\\
+(-1)^{i}\int_{t_1}^{t_2}h^{(m)}(t)
\cdot\Biggl( \underbrace{\int_{t_2-\tau}^{t}\int_{t_2-\tau}^{s_1}\dots
\int_{t_2-\tau}^{s_{m-i-1}}}_{m-i~\textnormal{times}}\Bigl(\partial_{i+2}
L[q]^m_{\tau}(s_{m-i})\Bigr)ds_{m-i} \dots ds_{2}ds_1\Biggr)dt\Biggr\}
\end{multline}
and
\begin{multline}
\label{eq:identity2}
\sum\limits_{i=0}^{m}\int_{t_1}^{t_2-\tau}
\partial_{i+m+3}L[q]^m_{\tau}(t+\tau)\cdot h^{(i)}(t) dt\\
=\sum\limits_{i=0}^{m}\Biggl\{\Biggl[\sum\limits_{j=1}^{m-i}(-1)^{j+1}h^{(i+j-1)}(t)
\cdot\Biggl(\underbrace{\int_{t_2-\tau}^{t}\int_{t_2-\tau}^{s_1}
\dots \int_{t_2-\tau}^{s_{j-1}}}_{j~\textnormal{times}}\Bigl(\partial_{i+m+3}
L[q]^m_{\tau}(s_j+\tau)\Bigr)ds_{j}\dots ds_2 ds_1\Biggr)\Biggr]_{t_1}^{t_2-\tau}\\
+(-1)^{i}\int_{t_1}^{t_2-\tau}h^{(m)}(t)
\cdot\Biggl( \underbrace{\int_{t_2-\tau}^{t}\int_{t_2-\tau}^{s_1}
\dots \int_{t_2-\tau}^{s_{m-i-1}}}_{m-i~\textnormal{times}}\Bigl(\partial_{i+m+3}
L[q]^m_{\tau}(s_{m-i}+\tau)\Bigr)ds_{m-i} \dots ds_{2}ds_1\Biggr)dt\Biggr\}.
\end{multline}
Because $h^{(i)}(t_2)=0,\,i=0,\ldots,m-1$, and $h(t)=0$, $t \in[t_{1}-\tau,t_{1}]$,
the terms without integral sign in the right-hand sides of identities \eqref{eq:identity1}
and \eqref{eq:identity2} vanish. Therefore, equation \eqref{pel1} becomes
\begin{multline}
\label{eq:identity3}
0=\int_{t_1}^{t_2-\tau}h^{(m)}(t)
\cdot\Biggl[\sum\limits_{i=0}^{m}(-1)^{i}
\Biggl( \underbrace{\int_{t_2-\tau}^{t}\int_{t_2-\tau}^{s_1}
\dots \int_{t_2-\tau}^{s_{m-i-1}}}_{m-i~\textnormal{times}}\Bigl(\partial_{i+2}
L[q]^m_{\tau}(s_{m-i})\\
+\partial_{i+m+3}
L[q]^m_{\tau}(s_{m-i}+\tau)\Bigr)ds_{m-i} \dots ds_{2}ds_1\Biggr)\Biggr]dt\\
+\int_{t_2-\tau}^{t_2}h^{(m)}(t)
\cdot\Biggl[\sum\limits_{i=0}^{m}(-1)^{i}
\Biggl( \underbrace{\int_{t_2-\tau}^{t}\int_{t_2-\tau}^{s_1}
\dots \int_{t_2-\tau}^{s_{m-i-1}}}_{m-i~\textnormal{times}}\Bigl(\partial_{i+2}
L[q]^m_{\tau}(s_{m-i})\Bigr)ds_{m-i} \dots ds_{2}ds_1\Biggr)\Biggr]dt.
\end{multline}
For $i=0,\dots,m$ we define functions
\begin{equation*}
\varphi_i (t)=
\begin{cases}
\partial_{i+2}L[q]^m_{\tau}(t)+\partial_{i+m+3}L[q]^m_{\tau}(t+\tau)
& ~\textnormal{for}~ t_1\leq t\leq t_2-\tau\\
\partial_{i+2}L[q]^m_{\tau}(t) & ~\textnormal{for}~ t_2-\tau\leq t\leq t_2.
\end{cases}
\end{equation*}
Then one can write equation \eqref{eq:identity3} as follows:
\begin{equation*}
0=\int_{t_1}^{t_2}h^{(m)}(t)
\cdot\Biggl[\sum\limits_{i=0}^{m}(-1)^{i}
\Biggl( \underbrace{\int_{t_2-\tau}^{t}\int_{t_2-\tau}^{s_1}
\dots \int_{t_2-\tau}^{s_{m-i-1}}}_{m-i~\textnormal{times}}
\Bigl(\varphi_i (s_{m-i})\Bigr)ds_{m-i} \dots ds_{2}ds_1\Biggr)\Biggr]dt.
\end{equation*}
Applying the higher-order DuBois--Reymond lemma \cite{Jost:book,Troutman:book},
one arrives to \eqref{eq:ELdeordmInt} and \eqref{eq:ELdeordmInt1}.
\end{proof}

\begin{corollary}[Higher-order Euler--Lagrange equations with time delay in differential form]
\label{cor:16}
If $q(\cdot)\in\mathbb{W}^{m,\infty}\left([t_1-\tau,t_2], \mathbb{R}^{n}\right)$
is an extremal of functional \eqref{Pm}, then
\begin{equation}
\label{eq:ELdeordm}
\sum_{i=0}^{m}(-1)^{i}\frac{d^{i}}{dt^{i}}\Bigl(\partial_{i+2}
L[q]^m_{\tau}(t)+\partial_{i+m+3}
L[q]^m_{\tau}(t+\tau)\Bigr)=0
\end{equation}
for $t_1\leq t\leq t_{2}-\tau$ and
\begin{equation}
\label{eq:ELdeordm1}
\sum_{i=0}^{m}(-1)^{i}\frac{d^{i}}{dt^{i}}\partial_{i+2}
L[q]^m_{\tau}(t)=0
\end{equation}
for $t_{2}-\tau\leq t \leq t_{2}$.
\end{corollary}

\begin{proof}
We obtain \eqref{eq:ELdeordm} and \eqref{eq:ELdeordm1} applying the derivative of order
$m$ to \eqref{eq:ELdeordmInt} and \eqref{eq:ELdeordmInt1}, respectively.
\end{proof}

\begin{remark}
If $m=1$, then the higher-order Euler--Lagrange equations
\eqref{eq:ELdeordm}--\eqref{eq:ELdeordm1} reduce to \eqref{EL1}.
\end{remark}

Associated to a given function $q(\cdot)\in\mathbb{W}^{m,\infty}\left([t_1-\tau,t_2],
\mathbb{R}^{n}\right)$, it is convenient to introduce
the following quantities (\textrm{cf.} \cite{Torres:proper}):
\begin{equation}
\label{eq:eqprin}
\psi^{j}_1=\sum_{i=0}^{m-j}(-1)^{i}\frac{d^{i}}{dt^{i}}\Bigl(\partial_{i+j+2}
L[q]^m_{\tau}(t)+\partial_{i+j+m+3}
L[q]^m_{\tau}(t+\tau)\Bigr)
\end{equation}
for $t_1\leq t\leq t_{2}-\tau$, and
\begin{equation}
\label{eq:eqprin11}
\psi^{j}_2=\sum_{i=0}^{m-j}(-1)^{i}\frac{d^{i}}{dt^{i}}\partial_{i+j+2}
L[q]^m_{\tau}(t)
\end{equation}
for $t_{2}-\tau\leq t\leq t_{2}$, where $j=0,\ldots,m$.
These operators are useful for our purposes because of the following properties:
\begin{equation}
\label{eq:eqprin1}
\frac{d}{dt}\psi^{j}_1=\partial_{j+1}L[q]^m_{\tau}(t)
+\partial_{j+m+2}L[q]^m_{\tau}(t+\tau)-\psi^{j-1}_1
\end{equation}
for $t_1\leq t\leq t_{2}-\tau$, and
\begin{equation*}
\frac{d}{dt}\psi^{j}_2=\partial_{j+1}L[q]^m_{\tau}(t)
-\psi^{j-1}_2
\end{equation*}
for $t_{2}-\tau\leq t\leq t_{2}$, where $j=1,\ldots,m$.
We are now in conditions to prove a higher-order DuBois--Reymond
optimality condition for problems with time delay.

\begin{theorem}[Higher-order delayed DuBois--Reymond condition]
\label{theo:cDRifm}
If $q(\cdot)\in\mathbb{W}^{m,\infty}\left([t_1-\tau,t_2], \mathbb{R}^{n}\right)$
is an extremal of functional \eqref{Pm}, then
\begin{equation}
\label{eq:DBRordm}
\frac{d}{dt}\left(L[q]^m_{\tau}(t)
-\sum_{j=1}^{m}\psi^{j}_1\cdot q^{(j)}(t)\right)
=\partial_{1}
L[q]^m_{\tau}(t)
\end{equation}
for $t_1\leq t\leq t_{2}-\tau$ and
\begin{equation}
\label{eq:DBRordm:2}
\frac{d}{dt}\left(L[q]^m_{\tau}(t)
-\sum_{j=1}^{m}\psi^{j}_2\cdot q^{(j)}(t)\right)
=\partial_{1}
L[q]^m_{\tau}(t)
\end{equation}
for $t_{2}-\tau\leq t\leq t_{2}$, where $\psi^{j}_1$ is given
by \eqref{eq:eqprin} and $\psi^{j}_2$ by \eqref{eq:eqprin11}.
\end{theorem}

\begin{proof}
We prove the theorem in the interval $t_{1}\leq t\leq t_{2}-\tau$.
The proof is similar for $t_{2}-\tau\leq t\leq t_{2}$.
We derive equation \eqref{eq:DBRordm} as follows:
\begin{multline}
\label{pr}
\int_{t_1}^{t_2}\frac{d}{dt}\left(L[q]^m_{\tau}(t)
-\sum_{j=1}^{m}\psi^{j}_1\cdot q^{(j)}(t)\right)dt\\
=\int_{t_1}^{t_2}\left(\partial_1 L[q]^{m}_{\tau}(t)+\sum_{j=0}^{m}
\partial_{j+2} L[q]^m_{\tau}(t)\cdot q^{(j+1)}(t)
-\sum_{j=1}^{m}\left(\dot{\psi}^{j}_1\cdot
q^{(j)}(t)+\psi^{j}_1\cdot q^{(j+1)}(t)\right)\right)dt\\
+\int_{t_1}^{t_2}\sum_{j=0}^{m}
\partial_{j+m+3} L[q]^m_{\tau}(t)\cdot q^{(j+1)}(t-\tau)dt.
\end{multline}
From \eqref{eq:eqprin1} and by performing a linear change
of variables $t=\sigma+\tau$ in the last integral
of \eqref{pr}, in the interval where $t_{1}\leq t\leq t_{2}-\tau$,
the equation \eqref{pr} becomes
\begin{multline}
\label{pr2}
\int_{t_1}^{t_2-\tau}\frac{d}{dt}\left(L[q]^m_{\tau}(t)
-\sum_{j=1}^{m}\psi^{j}_1\cdot q^{(j)}(t)\right)dt
=\int_{t_1}^{t_2-\tau}\left[\partial_1 L[q]_{\tau}^m(t)+\sum_{j=0}^{m}
\partial_{j+2} L[q]^m_{\tau}(t)\cdot q^{(j+1)}(t)\right.\\ \left.
-\sum_{j=1}^{m}\left(\left(\partial_{j+1}L[q]^m_{\tau}(t)
+\partial_{j+m+2}L[q]^m_{\tau}(t+\tau)-\psi^{j-1}_1\right)\cdot
q^{(j)}(t)+\psi^{j}_1\cdot q^{(j+1)}(t)\right)\right.\\
\left. +\sum_{j=0}^{m} \partial_{j+m+3} L[q]^m_{\tau}(t+\tau)
\cdot q^{(j+1)}(t)\right]dt.
\end{multline}
We now simplify the second term on the right-hand side of \eqref{pr2}:
\begin{multline}
\label{pr3}
\sum_{j=1}^{m}\left(\left(\partial_{j+1}L[q]^m_{\tau}(t)
+\partial_{j+m+2}L[q]^m_{\tau}(t+\tau)-\psi^{j-1}_1\right)\cdot
q^{(j)}(t)+\psi^{j}_1\cdot q^{(j+1)}(t)\right)\\
=\sum_{j=0}^{m-1}\Bigl(\left(\partial_{j+2}L[q]^m_{\tau}(t)
+\partial_{j+m+3}L[q]^m_{\tau}(t+\tau)-\psi^{j}_1\right)
\cdot q^{(j+1)}(t)+\psi^{j+1}_1\cdot q^{(j+2)}(t)\Bigr)\\
=\sum_{j=0}^{m-1}\left[\left(\partial_{j+2}L[q]^m_{\tau}(t)
+\partial_{j+m+3}L[q]^m_{\tau}(t+\tau)\right)\cdot q^{(j+1)}\right]
-\psi^{0}_1\cdot \dot{q}(t)+\psi^{m}_1\cdot q^{(m+1)}(t).
\end{multline}
Substituting \eqref{pr3} into \eqref{pr2} and using the higher-order Euler--Lagrange
equations with time delay \eqref{eq:ELdeordm}, and since, by definition,
$\psi^{m}_1=\partial_{m+2}L[q]^m_{\tau}(t)+\partial_{2m+3}L[q]^m_{\tau}(t+\tau)$
and
\begin{equation*}
\psi^0_1=
\sum_{i=0}^{m}(-1)^{i}\frac{d^{i}}{dt^{i}}\Bigl(\partial_{i+2}
L[q]^m_{\tau}(t)+\partial_{i+m+3}
L[q]^m_{\tau}(t+\tau)\Bigr)=0\, ,
\end{equation*}
we obtain the intended result, that is,
\begin{multline*}
\int_{t_1}^{t_2-\tau}\frac{d}{dt}\left(L[q]^m_{\tau}(t)
-\sum_{j=1}^{m}\psi^{j}_1\cdot q^{(j)}(t)\right)dt\\
=\int_{t_1}^{t_2-\tau}\left[\partial_1 L[q]_{\tau}^m(t)
+\left(\partial_{m+2}L[q]^m_{\tau}(t)+\partial_{2m+3}
L[q]^m_{\tau}(t+\tau)\right)\cdot q^{(m+1)}\right.\\
\left. +\psi^{0}_1\cdot
\dot{q}(t)-\psi^{m}_1\cdot q^{(m+1)}(t)\right]dt
=\int_{t_1}^{t_2-\tau}\partial_1 L[q]_{\tau}^m(t)dt.
\end{multline*}
\end{proof}

In the particular case when $m=1$, we obtain from Theorem~\ref{theo:cDRifm}
an extension of Theorem~\ref{theo:cdrnd} to the class of Lipschitz functions.

\begin{corollary}[Nonsmooth DuBois--Reymond conditions]
\label{cor:DR:m1}
If $q(\cdot)\in Lip\left([t_1-\tau,t_2];\mathbb{R}^n\right)$
is an extremal of functional \eqref{Pe}, then the DuBois--Reymond
conditions with time delay \eqref{eq:cdrnd} hold true.
\end{corollary}

\begin{proof}
For $m=1$, condition \eqref{eq:DBRordm} is reduced to
\begin{equation}
\label{eq:DBRordm111}
\frac{d}{dt}\left(L[q]_{\tau}(t)
-\psi^{1}_1\cdot \dot{q}(t)\right)
=\partial_{1} L[q]_{\tau}(t)
\end{equation}
for $t_1\leq t\leq t_{2}-\tau$, and \eqref{eq:DBRordm:2} to
\begin{equation}
\label{eq:DBRordm121}
\frac{d}{dt}\left(L[q]_{\tau}(t)
-\psi^{1}_2\cdot \dot{q}(t)\right)
=\partial_{1} L[q]_{\tau}(t)
\end{equation}
for $t_{2}-\tau\leq t\leq t_{2}$.
Keeping in mind  \eqref{eq:eqprin} and \eqref{eq:eqprin11},
we obtain
\begin{equation}
\label{eq:DBRord3}
\psi^{1}_1=\partial_3 L[q]_{\tau}(t)+\partial_5 L[q]_{\tau}(t+\tau)
\end{equation}
and
\begin{equation}
\label{eq:DBRord4}
\psi^{1}_2=\partial_3 L[q]_{\tau}(t).
\end{equation}
One finds the intended equalities \eqref{eq:cdrnd} by substituting the quantities
\eqref{eq:DBRord3} and \eqref{eq:DBRord4} into \eqref{eq:DBRordm111}
and \eqref{eq:DBRordm121}, respectively.
\end{proof}

% ------------------------

\subsection{Higher-order Noether's symmetry theorem with time delay}

Now, we generalize the Noether-type theorem proved in Section~\ref{sec:MR}
to the more general case of delayed variational problems with higher-order derivatives.

\begin{definition}[Invariance of \eqref{Pm} up to a gauge-term]
\label{def:invaifm}
Consider the $s$-parameter group of infinitesimal transformations \eqref{eq:tinf}.
Functional \eqref{Pm} is invariant under \eqref{eq:tinf}
up to the gauge-term $\Phi$ if
\begin{multline}
\label{eq:invndm}
\int_{I} \dot{\Phi}[q]^m_{\tau}(t)dt
= \frac{d}{ds} \int_{\bar{t}(I)} L\left(\bar{t},\bar{q}(\bar{t}),
{\bar{q}}'(\bar{t}),\ldots,\bar{q}^{(m)}(\bar{t}),\right.\\
\left.\left.\bar{q}(\bar{t}-\tau),{\bar{q}}'(\bar{t}-\tau),
\ldots,\bar{q}^{(m)}(\bar{t}-\tau)\right)
(1+s\dot{\eta}(t,q(t))dt\right|_{s=0}
\end{multline}
for any  subinterval $I \subseteq [t_1,t_2]$ and for all
$q(\cdot)\in \mathbb{W}^{m,\infty}\left([t_1-\tau,t_2], \mathbb{R}^{n}\right)$.
\end{definition}

\begin{remark}
Expressions $\dot{\Phi}$ and $\bar{q}^{(i)}$ in equation \eqref{eq:invndm},
$i=1,\ldots,m$, are interpreted as
\begin{equation}
\label{eq:invifm1}
\dot{\Phi}=\frac{d}{dt}\Phi\,\,,\quad\bar{q}'
=\frac{d\bar{q}}{d\bar{t}}=\frac{\frac{d\bar{q}}{dt}}{\frac{d\bar{t}}{dt}}\,\,,
\quad \bar{q}^{(i)}=\frac{d^{i}\bar{q}}{d\bar{t}^{i}}=
\frac{\frac{d}{dt}\left(\frac{d^{i-1}}{d\bar{t}^{i-1}}\bar{q}\right)}{\frac{d\bar{t}}{dt}},\,
i=2,\ldots,m.
\end{equation}
\end{remark}

The next lemma gives a necessary condition of invariance for functional \eqref{Pm}.

\begin{lemma}[Necessary condition of invariance for \eqref{Pm}]
\label{thm:cnsi}
If functional \eqref{Pm} is invariant up to the gauge-term $\Phi$ under
the $s$-parameter group of infinitesimal transformations \eqref{eq:tinf}, then
\begin{multline}
\label{eq:cnsiifm}
\int_{t_1}^{t_2-\tau}\left[-\dot{\Phi}[q]^m_{\tau}(t)+\partial_{1}
L[q]^m_{\tau}(t)\eta(t,q)+ L[q]^m_{\tau}(t)
\dot{\eta}(t,q)\right.\\ \left.
+\sum_{i=0}^{m}\left(\partial_{i+2}
L[q]^m_{\tau}(t)+\partial_{i+m+3}
L[q]^m_{\tau}(t+\tau)\right)\cdot
\rho^{i}(t) \right]dt =0
\end{multline}
for $t_1\leq t\leq t_{2}-\tau$ and
\begin{equation}
\label{eq:cnsiifm11}
\int_{t_2-\tau}^{t_2}\left[-\dot{\Phi}[q]^m_{\tau}(t)+\partial_{1}
L[q]^m_{\tau}(t)\eta(t,q)+ L[q]^m_{\tau}(t)
\dot{\eta}(t,q) + \sum_{i=0}^{m}\partial_{i+2}
L[q]^m_{\tau}(t)\cdot \rho^{i}(t)\right]dt =0
\end{equation}
for $t_2-\tau\leq t\leq t_{2}$, where
\begin{equation}
\label{eq:cnsiifm1}
\begin{cases}
\rho^{0}(t)=\xi(t,q) \, , \\
\rho^{i}(t)=\frac{d}{dt}\left(\rho^{i-1}(t)\right)-q^{(i)}(t)\dot{\eta}(t,q)\, ,
\quad i=1,\ldots,m.
\end{cases}
\end{equation}
\end{lemma}

\begin{proof}
Without loss of generality, we take $I=[t_1,t_2]$.
Then, \eqref{eq:invndm} is equivalent to
\begin{multline}\label{mm}
\int_{t_1}^{t_2}\left[-\dot{\Phi}[q]_{\tau}(t)+\partial_{1}L[q]^m_{\tau}(t)\eta(t,q)
+\sum_{i=0}^{m}\partial_{i+2} L[q]^m_{\tau}(t)\cdot\frac{\partial}{\partial s}
\left.\left(\frac{d^{i}\bar{q}}{d\bar{t}^{i}}\right)\right|_{s=0}\right.\\ \left.
+\sum_{i=0}^{m}\partial_{i+m+3} L[q]^m_{\tau}(t)\cdot\frac{\partial}{\partial s}
\left.\left(\frac{d^{i}\bar{q}(\bar{t}-\tau)}{d(\bar{t}-\tau)^{i}}\right)\right|_{s=0}
+L[q]^m_{\tau}(t) \dot{\eta} \right] =0.
\end{multline}
Using the fact that \eqref{eq:invifm1} implies
\begin{equation*}
\frac{\partial}{\partial s}\left.\left(\frac{d\bar{q}(\bar{t})}{d\bar{t}}\right)\right|_{s=0}
=\dot{\xi}(t,q)-\dot{q}\dot{\eta}(t,q) \, ,
\end{equation*}
\begin{equation*}
\frac{\partial}{\partial s}\left.\left(
\frac{d^{i}\bar{q}(\bar{t})}{d\bar{t}^{i}}\right)\right|_{s=0}
=\frac{d}{dt}\left[\frac{\partial}{\partial s}
\left.\left(\frac{d^{i-1}\bar{q}(\bar{t})}{d\bar{t}^{i-1}}\right)\right|_{s=0}\right]
-q^{(i)}(t)\dot{\eta}(t,q)\, , \quad i=2,\ldots,m,
\end{equation*}
then equation \eqref{mm} becomes
\begin{multline}
\label{mm1}
\int_{t_1}^{t_2}\left[-\dot{\Phi}[q]^m_{\tau}(t)+\partial_{1}
L[q]^m_{\tau}(t)\eta(t,q)+ L[q]^m_{\tau}(t)
\dot{\eta}(t,q)\right.\\ \left.
+\sum_{i=0}^{m}\partial_{i+2}
L[q]^m_{\tau}(t)\cdot\rho^{i}(t)+\sum_{i=0}^{m}\partial_{i+m+3}
L[q]^m_{\tau}(t)\cdot
\rho^{i}(t-\tau)
\right]dt =0.
\end{multline}
Performing the linear change of variables $t=\sigma+\tau$ in the last integral
of \eqref{mm1}, and keeping in mind that $L[q]^m_{\tau}(t)\equiv 0$ on
$[t_1-\tau,t_1]$, equation \eqref{mm1} becomes
\begin{multline}
\label{mm2}
\int_{t_1}^{t_2-\tau}\left[-\dot{\Phi}[q]^m_{\tau}(t)+\partial_{1}
L[q]^m_{\tau}(t)\eta(t,q)+ L[q]^m_{\tau}(t)
\dot{\eta}(t,q)\right.\\ \left.
+\sum_{i=0}^{m}\left(\partial_{i+2}
L[q]^m_{\tau}(t)+\partial_{i+m+3}
L[q]^m_{\tau}(t+\tau)\right)\cdot
\rho^{i}(t) \right]dt \\
+ \int_{t_2-\tau}^{t_2}\left[-\dot{\Phi}[q]^m_{\tau}(t)+\partial_{1}
L[q]^m_{\tau}(t)\eta(t,q)+ L[q]^m_{\tau}(t)
\dot{\eta}(t,q) +\sum_{i=0}^{m}\partial_{i+2} L[q]^m_{\tau}(t)\cdot
\rho^{i}(t)\right]dt =0.
\end{multline}
Equations \eqref{eq:cnsiifm} and \eqref{eq:cnsiifm11} follow from the fact
that \eqref{mm2} holds for an arbitrary $I \subseteq [t_1,t_2]$.
\end{proof}

\begin{definition}[Higher-order constant of motion/conservation law with time delay]
\label{def:leicond2}
A quantity
\begin{multline*}
C\{q\}_{\tau}^{m}(t) := C\Bigl(t,t+\tau,q(t),\dot{q}(t),
\ldots,q^{(m)}(t),q(t-\tau),\dot{q}(t-\tau),
\ldots,q^{(m)}(t-\tau),\\q(t+\tau),\dot{q}(t+\tau),
\ldots,q^{(m)}(t+\tau)\Bigr)
\end{multline*}
is a higher-order constant of motion with time delay $\tau$ if
\begin{equation}
\label{eq:conslaw:td}
\frac{d}{dt} C\{q\}_{\tau}^{m}(t) = 0,
\end{equation}
$t\in [t_1,t_2]$, along any
$q(\cdot)\in \mathbb{W}^{m,\infty}\left([t_1-\tau,t_2], \mathbb{R}^{n}\right)$
satisfying both Theorem~\ref{Thm:ELdeordm} and Theorem~\ref{theo:cDRifm}.
The equality \eqref{eq:conslaw:td} is then said
to be a higher-order conservation law with time delay.
\end{definition}

\begin{theorem}[Higher-order Noether's symmetry theorem with time delay]
\label{thm:Noether}
If functional \eqref{Pm} is invariant up to the gauge-term $\Phi$
in the sense of Definition~\ref{def:invaifm},
then the quantity $C\{q\}_{\tau}^{m}(t)$ defined by
\begin{equation}
\label{eq:TeNetm}
\sum_{j=1}^{m}\psi^{j}_1\cdot
\rho^{j-1}(t)+\left(L[q]^m_{\tau}(t)
-\sum_{j=1}^{m}\psi^{j}_1\cdot q^{(j)}(t)\right)\eta(t,q)
-\Phi[q]^m_{\tau}(t)
\end{equation}
for $t_1\leq t\leq t_{2}-\tau$ and by
\begin{equation*}
\label{eq:TeNetm1}
\sum_{j=1}^{m}\psi^{j}_2\cdot
\rho^{j-1}(t)+\left(L[q]^m_{\tau}(t)
-\sum_{j=1}^{m}\psi^{j}_2\cdot q^{(j)}(t)\right)\eta(t,q)
-\Phi[q]^m_{\tau}(t)
\end{equation*}
for $t_2-\tau\leq t\leq t_{2}$, is a higher-order constant of motion with time delay
(\textrm{cf.} Definition~\ref{def:leicond2}), where $\psi^{j}_1$ and $\psi^{j}_2$
are given by \eqref{eq:eqprin} and \eqref{eq:eqprin11}, respectively.
\end{theorem}

\begin{proof}
We prove the theorem in the interval $t_1\leq t\leq t_{2}-\tau$.
The proof is similar in the interval $t_2-\tau\leq t\leq t_{2}$.
Equation \eqref{eq:TeNetm} follows by direct calculations:
\begin{equation}
\label{eq:TeNetm2}
\begin{split}
0&=\int_{t_1}^{t_2-\tau}\frac{d}{dt}\left[\psi_1^1\cdot\rho^0+\sum_{j=2}^{m}\psi^{j}_1
\cdot \rho^{j-1}(t) +\left(L[q]^m_{\tau}(t)
-\sum_{j=1}^{m}\psi^{j}_1\cdot q^{(j)}(t)\right)\eta(t,q)
-\Phi[q]^m_{\tau}(t)\right]dt\\
&=\int_{t_1}^{t_2-\tau}\left[-\dot{\Phi}[q]^m_{\tau}(t)+\rho^0(t)\cdot\frac{d}{dt}\psi^1_1
+\psi^1_1\cdot\frac{d}{dt}\rho^0(t)
+\sum_{j=2}^{m}\left(\rho^{j-1}(t)\cdot\frac{d}{dt}\psi^{j}_1
+\psi^{j}_1\cdot\frac{d}{dt}\rho^{j-1}(t)\right)\right.\\
& \qquad \left.+\eta(t,q)\frac{d}{dt}\left(L[q]^m_{\tau}(t)-\sum_{j=1}^{m}\psi_1^{j}
\cdot q^{(j)}(t)\right)+\left(L[q]^m_{\tau}(t)-\sum_{j=1}^{m}\psi_1^{j}
\cdot q^{(j)}(t)\right)\dot{\eta}(t,q)\right]dt.
\end{split}
\end{equation}
Using the Euler--Lagrange equation \eqref{eq:ELdeordm}, the
DuBois--Reymond condition \eqref{eq:DBRordm}, and relations
\eqref{eq:eqprin1}  and \eqref{eq:cnsiifm1} in \eqref{eq:TeNetm2},
we obtain:
\begin{multline}
\label{eq:dems}
\int_{t_1}^{t_2-\tau}\left[-\dot{\Phi}[q]^m_{\tau}(t)
+\left(\partial_{2} L[q]^m_{\tau}(t)+\partial_{m+3}
L[q]^m_{\tau}(t+\tau)\right)\cdot\xi(t,q)+\psi^{1}_1
\cdot(\rho^1(t)+\dot{q}(t)\dot{\tau}(t,q))\right. \\
\left. +\sum_{j=2}^{m}\left[\left(\partial_{j+1}
L[q]^m_{\tau}(t)+\partial_{j+m+2}
L[q]^m_{\tau}(t+\tau)-\psi_1^{j-1}\right)\cdot\rho^{j-1}(t)+
\psi_1^{j}\cdot\left(\rho^{j}(t)+q^{(j)}(t)
\dot{\tau}(t,q)\right)\right]\right.\\ \left.
+\partial_{1} L[q]^m_{\tau}(t)\eta(t,q)
+\left(L[q]^m_{\tau}(t)-\sum_{j=1}^{m}\psi_1^{j}
\cdot q^{(j)}(t)\right)\dot{\eta}(t,q)\right]dt\\
=\int_{t_1}^{t_2-\tau}\Bigl[
\partial_{1} L [q]^m_{\tau}(t)\eta(t,q)+L[q]^m_{\tau}(t)\dot{\eta}(t,q)
+\left(\partial_{2} L[q]^m_{\tau}(t)+\partial_{m+3}
L[q]^m_{\tau}(t+\tau)\right)\cdot\xi(t,q)\\
+\psi_1^{1} \cdot(\rho^1(t)+\dot{q}(t)\dot{\eta}(t,q))-\psi_1^1\cdot\rho^1(t)
-\psi_1^1\cdot\dot{q}(t)\dot{\eta}(t,q)+\psi_1^m\cdot\rho^m(t)\\
+\sum_{j=2}^{m}\left(\partial_{j+1} L[q]^m_{\tau}(t)+\partial_{j+m+2}
L[q]^m_{\tau}(t+\tau)\right]\cdot\rho^{j-1}(t)-\dot{\Phi}[q]^m_{\tau}(t)\Bigr]dt = 0.
\end{multline}
Simplification of \eqref{eq:dems} leads to the necessary
condition of invariance \eqref{eq:cnsiifm}.
\end{proof}

% ------------------------

\section*{Acknowledgements}

This work was supported by FEDER funds through
COMPETE --- Operational Programme Factors of Competitiveness
(``Programa Operacional Factores de Competitividade'')
and by Portuguese funds through the
Center for Research and Development
in Mathematics and Applications (University of Aveiro)
and the Portuguese Foundation for Science and Technology
(``FCT --- Funda\c{c}\~{a}o para a Ci\^{e}ncia e a Tecnologia''),
within project PEst-C/MAT/UI4106/2011
with COMPETE number FCOMP-01-0124-FEDER-022690.
Frederico was also supported by FCT through the
post-doc fellowship SFRH/BPD/51455/2011,
program ``Ci\^{e}ncia Global'', Odzijewicz and Torres
by EU funding under the 7th Framework Programme
FP7-PEOPLE-2010-ITN, grant agreement number 264735-SADCO.

% ------------------------

% ------------------------

\end{document}